\documentclass[a4paper]{amsart}

\usepackage{amsmath}
\usepackage{amsfonts}
\usepackage{amsthm}
\usepackage{graphicx}
\usepackage{eucal}
\usepackage{amscd}
\usepackage[all,2cell]{xy}
\usepackage{amssymb}
\usepackage{mathrsfs}

 \usepackage{tikz}
 \usetikzlibrary{positioning}
 \tikzset{mynode/.style={draw,circle,inner sep=1pt,outer sep=0pt}}

\newtheorem{teo}{Theorem}[section]
\newtheorem{cor}[teo]{Corollary}
\newtheorem{lem}[teo]{Lemma}
\newtheorem{defi}[teo]{Definition}

\newtheorem{remark}[teo]{Remark}

\newcommand{\X}{\ensuremath{\mathbb{X}}}

\newdir{ |>}{{}*!/-3.5pt/@{|}*!/-8pt/:(1,-.2)@^{>}*!/-8pt/:(1,+.2)@_{>}}

\dedicatory{Dedicated to the memory of V$\check{e}$ra Trnkov\'a }

\begin{document}

\title{SPLIT EXTENSIONS AND SEMIDIRECT PRODUCTS OF UNITARY MAGMAS}

\author{Marino Gran}
\address[Marino Gran]{Institut de Recherche en Math\'ematique et Physique, Universit\'e catholique de Louvain, Chemin du Cyclotron 2, 1348 Louvain-la-Neuve, Belgium}
\thanks{}
\email{marino.gran@uclouvain.be}

\author{George Janelidze}
\address[George Janelidze]{Department of Mathematics and Applied Mathematics, University of Cape Town, Rondebosch 7700, South Africa}
\thanks{Partially supported by South African NRF}
\email{george.janelidze@uct.ac.za}

\author{Manuela Sobral}
\address[Manuela Sobral]{CMUC and Departamento de
Matem\'atica, Universidade de Coimbra, 3001--501 Coimbra,
Portugal}
\thanks{
Partially supported by the Centre for Mathematics of
the University of Coimbra -- UID/MAT/00324/2019}
\email{sobral@mat.uc.pt}



\begin{abstract}
We develop a theory of split extensions of unitary magmas, which includes defining such extensions and describing them via suitably defined semidirect product, yielding an equivalence between the categories of split extensions and of (suitably defined) actions of unitary magmas on unitary magmas. The class of split extensions is pullback stable but not closed under composition. We introduce two subclasses of it that have both of these properties.
\end{abstract}
\date{\today}
\maketitle

\section*{Introduction}

As one knows from any first course of group theory, every split epimorphism of groups is, up to isomorphism, a semidirect product projection. A more refined categorical formulation says that there is a category equivalence between:
\begin{itemize}
	\item the category of split extensions
	\begin{equation*}\label{split-extension}
	\xymatrix{X  \ar[r]^-{\kappa}   & A  \ar[r]_-{\alpha} & B \ar@<-1ex>[l]_-{\beta}}
	\end{equation*}
	of groups, and
	\item the category of group actions.
\end{itemize}
Here a split extension is a diagram above in the category of groups with $\alpha\beta=1$ and $\kappa$ being a kernel of $\alpha$, while a group action is a triple $(B,X,h)$, in which $B$ is a group acting on a group $X$, and $h:B\times X\to X$ is its action. Under this equivalence, an action $(B,X,h)$ corresponds to the split extension
\begin{equation*}\label{semidirect}
\xymatrix{X  \ar[r]^-{\iota_1}   & X \rtimes B \ar[r]_-{\pi_2} & B \ar@<-1ex>[l]_-{\iota_2},}
\end{equation*}
involving the semidirect product $X \rtimes B$, in which $\iota_1$, $\iota_2$, and $\pi_2$ are defined by $\iota_1(x)=(x,0)$, $\iota_2(b)=(0,b)$, and $\pi_2(x,b)=b$, respectively (in additive notation).\\

\textit{We shall refer to this category equivalence as the theory of split extensions of groups.}\\

Based on Bourn's theory of \textit{protomodular categories} (see \cite{[B]} and \cite{[BB]}) and the theory of monads, the theory of split extensions was extended to the context of abstract \textit{semi-abelian categories} in the sense of \cite{[JMT]} (see \cite{[BJ]} and \cite{[BJK]}). This widely generalizes not only the group case but also more general algebraic cases considered by Orzech \cite{[O]} and Porter \cite{[P]}.

Going beyond the semi-abelian context is also possible, but then split extensions should be defined differently, involving an additional structure and properties. Specifically, in the case of monoids (see \cite{[BMMS]} and references therein), a split extension (called a \textit{Schreier split extension}) should be defined as a diagram
\begin{equation*}\label{split-extension}
\xymatrix{X  \ar[r]_-{\kappa}   & A  \ar[r]_-{\alpha} \ar@<-1ex>[l]_-{\lambda} & B \ar@<-1ex>[l]_-{\beta},}
\end{equation*}
in which  $\alpha$, $\beta$, and $\kappa$ are as before, and $\lambda$ is a map (not necessarily a monoid homomorphism) with $\kappa\lambda+\beta\alpha=1$ and $\lambda(\kappa(x)+\beta(b))=x$, for all $x\in X$ and $b\in B$ (using additive notation again). The context of monoids is not more general than the semi-abelian one of course, but it easily extends to cover the situations considered in \cite{[O]} and \cite{[P]} (see \cite{[MMS]}). Note that Schreier split extensions of monoids had implicitly been used in the work of A. Patchkoria on Schreier internal categories in monoids \cite{[AP]}. \\

The purpose of the present paper is to develop another generalization of the theory of split extensions, namely from monoids to unitary magmas, that is, to algebraic structures of the form $M=(M,0,+)$, where the only axiom required is $0+x=x=x+0$ for all $x\in M$. It turns out that this is indeed possible (see Theorem 2.8) with various surprising and less surprising additional observations, some of which are:
\begin{itemize}
	\item Defining a split extension one needs not just to involve a map denoted by $\lambda$ above, but also to require unusual new conditions of partial associativity. These conditions are harmless in the sense that every two unitary magmas $B$ and $X$ admit a split extension involving them. These conditions are the equalities (15)-(17) in Definition 1.4 (although (15) and (17) in fact imply (16)).
	\item On the other hand, defining an action of $B$ on $X$ we do not require any properties involving addition (see Definition 1.1). Still, requiring such properties would be natural, and they give a nicer classes of split extensions/epimorphisms (see Definition 3.5, Corollary 3.6, Theorem 3.7, and Subsection 4.3).
	\item The above-mentioned ``nicer classes'' are not only pullback stable in the class of all split extensions, but also closed under composition (see Subsection 4.3 again).
	\item Just as for groups and monoids, all our definitions, constructions, and results can be copied for internal unitary magmas in abstract categories (see Subsection 4.6).
	\item Just as for monoids, our split extensions satisfy some forms of Short Five Lemma (see Subsection 4.4), but not the strongest form valid for groups (see 4.4(c)), since we are generalizing the monoid case.
\end{itemize}

\textit{\textbf{Convention}: All magmas we consider below are supposed to be unitary and all maps between them are supposed to preserve zero, but not necessarily addition, unless stated otherwise.}

\section{Introducing semidirect products and split extensions}

\begin{defi}
	Let $B$ and $X$ be magmas. A map $B\times X\to X$ written as $(b,x)\mapsto bx$ is said to be an action of $B$ on $X$ if
	\begin{equation}
	0x=x,\,\,b0=0,
	\end{equation}
	for all $x\in X$ and $b\in B$.
\end{defi}
\begin{defi}
	For magmas $B$ and $X$ and an action of $B$ on $X$, the semidirect product diagram is the diagram
	\begin{equation}\label{semidirect}
	\xymatrix{X  \ar[r]_-{\iota_1}   & X \rtimes B \ar[r]_-{\pi_2} \ar@<-1ex>[l]_-{\pi_1} & B \ar@<-1ex>[l]_-{\iota_2}
	}
	\end{equation}
	in which:
	\begin{itemize}
		\item [(a)] $X \rtimes B$ is a magma whose underlying set is $X\times B$ and whose addition is defined by $(x,b)+(x',b')=(x+bx',b+b')$ (thus making $(0,0)$ its zero);
		\item [(b)] the maps involved are defined by $\iota_1(x)=(x,0)$, $\iota_2(b)=(0,b)$, $\pi_1(x,b)=x$, and $\pi_2(x,b)=b$.
	\end{itemize}
\end{defi}
\begin{lem}
	The semidirect product diagram $(2)$ satisfies the following conditions:
	\begin{itemize}
		\item [(a)] all maps involved in it preserve zero, and, moreover, the maps $\iota_1$, $\iota_2$, and $\pi_2$ are magma homomorphisms;
		\item [(b)] the equalities
		\begin{equation}\label{identity}
		\pi_1\iota_1=1,\,\,\pi_2\iota_2=1,
		\end{equation}
		\begin{equation}\label{0-map}
		\pi_1\iota_2=0,\,\,\pi_2\iota_1=0,
		\end{equation}
		\begin{equation}\label{biproduct}
		\iota_1\pi_1+\iota_2\pi_2=1,
		\end{equation}
		\begin{equation}\label{non-dist}
		\pi_1(\iota_1(x)+\iota_2(b))=x,
		\end{equation}
		\begin{equation}\label{left-m}
		(x,0)+((0,b)+(x',b'))=((x,0)+(0,b))+(x',b'),
		\end{equation}
		\begin{equation}\label{mid-m}
		(x,0)+((x',b')+(0,b))=((x,0)+(x',b'))+(0,b),
		\end{equation}
		\begin{equation}\label{right-m}
		(x',b')+((x,0)+(0,b))=((x',b')+(x,0))+(0,b)
		\end{equation}
		hold for all $x,x'\in X$ and $b,b'\in B$. \qed
	\end{itemize}
\end{lem}
\begin{defi}
	A split extension of magmas is a diagram
	
	\begin{equation}\label{split-extension}
	\xymatrix{X  \ar[r]_-{\kappa}   & A  \ar[r]_-{\alpha} \ar@<-1ex>[l]_-{\lambda} & B \ar@<-1ex>[l]_-{\beta}
	}
	\end{equation}
	
	in which:
	\begin{itemize}
		\item [(a)] $X$, $A$, and $B$ are magmas, $\alpha$, $\beta$, and $\kappa$ are magma homomorphisms, and $\lambda$ preserves zero;
		\item [(b)] the equalities
		\begin{equation}\label{split-identity}
		\lambda\kappa=1,\,\,\alpha\beta=1,
		\end{equation}
		\begin{equation}\label{split-0-map}
		\lambda\beta=0,\,\,\alpha\kappa=0,
		\end{equation}
		\begin{equation}\label{split-biproduct}
		\kappa\lambda+\beta\alpha=1,
		\end{equation}
		\begin{equation}\label{split-non-dis}
		\lambda(\kappa(x)+\beta(b))=x,
		\end{equation}
		\begin{equation}\label{split-left-m}
		\kappa(x)+(\beta(b)+a)=(\kappa(x)+\beta(b))+a,
		\end{equation}
		\begin{equation}\label{split-mid-m}
		\kappa(x)+(a+\beta(b))=(\kappa(x)+a)+\beta(b),
		\end{equation}
		\begin{equation}\label{split-right-m}
		a+(\kappa(x)+\beta(b))=(a+\kappa(x))+\beta(b),
		\end{equation}
		hold for all $x,x'\in X$, $a\in A$ and $b,b'\in B$.
	\end{itemize}
\end{defi}

\begin{defi}
	Given a split extension $\eqref{split-extension}$, the associated action of $B$ on $X$ is defined by $bx=\lambda(\beta(b)+\kappa(x))$.
\end{defi}

Consider the diagram

\begin{equation}\label{canonical-iso}
\xymatrix@=30pt{
	X  \ar@{=} [d] \ar[r]_-{\kappa}   & A  \ar[r]_-{\alpha} \ar@<-1ex>[l]_-{\lambda} \ar@<-0.5ex>[d]_{\varphi}  & B \ar@<-1ex>[l]_-{\beta}  \ar@{=}[d] \\X  \ar[r]_-{\iota_1}   & X \rtimes B \ar[r]_-{\pi_2} \ar@<-1ex>[l]_-{\pi_1} \ar@<-0.5ex>[u]_{\psi}  & B \ar@<-1ex>[l]_-{\iota_2}
}
\end{equation}

in which:
\begin{itemize}
	\item the top row is a split extension of magmas;
	\item the bottom row is a semidirect product diagram in which $B$ acts on $X$ as in Definition 1.5;
	\item $\varphi$ is defined by $\varphi(a)=(\lambda(a),\alpha(a))$;
	\item $\psi$ is defined by $\psi(x,b)=\kappa(x)+\beta(b)$.
\end{itemize}
We are going to prove several lemmas involving this diagram:
\begin{lem}
	The maps $\varphi$ and $\psi$ are bijections, inverse to each other.
\end{lem}
\begin{proof}
	For $x\in X$ and $b\in B$, we have:
	\begin{equation*}
	\varphi\psi(x,b)=\varphi(\kappa(x)+\beta(b))=(\lambda(\kappa(x)+\beta(b)),\alpha(\kappa(x)+\beta(b)))=(x,b),
	\end{equation*}
	where $\lambda(\kappa(x)+\beta(b))=x$ by \eqref{split-non-dis}, while $\alpha(\kappa(x)+\beta(b))=b$ follows from \eqref{split-0-map}, \eqref{split-identity},
	and the fact $\alpha$ is a magma homomorphism.
	
	For $a\in A$, we have:
	\begin{equation*}
	\psi\varphi(a)=\psi(\lambda(a),\alpha(a))=\kappa\lambda(a)+\beta\alpha(a)=a,
	\end{equation*}
	where the last equality holds by \eqref{split-biproduct}.
\end{proof}
\begin{lem}
	For $x\in X$ and $b\in B$, we have: $\beta(b)+\kappa(x)=\kappa(bx)+\beta(b)$.
\end{lem}
\begin{proof}
	Using Lemma 1.6, we calculate:
	\begin{equation*}
	\beta(b)+\kappa(x)=\psi\varphi(\beta(b)+\kappa(x))=\psi(\lambda(\beta(b)+\kappa(x)),\alpha(\beta(b)+\kappa(x)))=\psi(bx,b)
	\end{equation*}
	\begin{equation*}
	=\kappa(bx)+\beta(b),
	\end{equation*}
	where $\lambda(\beta(b)+\kappa(x))=bx$ by Definition 1.5, while $\alpha(\beta(b)+\kappa(x))=b$ follows from \eqref{split-identity}, \eqref{split-0-map}, and the fact $\alpha$ is a magma homomorphism.
\end{proof}
\begin{lem}
	The map $\psi$ is a magma homomorphism.
\end{lem}
\begin{proof}
	$\psi$ obviously preserves zero, and, to prove that it preserves addition, we calculate:
	\begin{equation*}
	\psi((x,b)+(x',b'))=\psi(x+bx',b+b')
	\end{equation*}
	(by 1.2(a))
	\begin{equation*}
	=\kappa(x+bx')+\beta(b+b')
	\end{equation*}
	(by definition of $\psi$)
	\begin{equation*}
	=\kappa(x+bx')+(\beta(b)+\beta(b'))
	\end{equation*}
	(since $\beta$ is a magma homomorphism)
	\begin{equation*}
	=(\kappa(x+bx')+\beta(b))+\beta(b')
	\end{equation*}
	(by (15), or by (16))
	\begin{equation*}
	=((\kappa(x)+\kappa(bx'))+\beta(b))+\beta(b')
	\end{equation*}
	(since $\kappa$ is a magma homomorphism)
	\begin{equation*}
	=(\kappa(x)+(\kappa(bx')+\beta(b)))+\beta(b')
	\end{equation*}
	(by (16), or by (17))
	\begin{equation*}
	=(\kappa(x)+(\beta(b)+\kappa(x')))+\beta(b')
	\end{equation*}
	(by Lemma 1.7)
	\begin{equation*}
	=((\kappa(x)+\beta(b))+\kappa(x'))+\beta(b')
	\end{equation*}
	(by (15))
	\begin{equation*}
	=(\kappa(x)+\beta(b))+(\kappa(x')+\beta(b'))
	\end{equation*}
	(by (17))
	\begin{equation*}
	=\psi(x,b)+\psi(x',b')
	\end{equation*}
	(by definition of $\psi$).
\end{proof}

\begin{remark}
	Lemmas 1.6 and 1.7 do not use the magma structure of $X \rtimes B$, and so one might prefer not to introduce Definitions 1.2 and 1.5 before Lemma 1.6, but to begin with Lemma 1.6 as the motivation of introducing the semidirect products as cartesian products with modified structure.
\end{remark}
\begin{lem} The diagram $(18)$ reasonably commutes, in the sense that
	\begin{equation}
	\varphi\kappa=\iota_1,\,\,\varphi\beta=\iota_2,\,\,\pi_1\varphi=\lambda,\,\,\pi_2\varphi=\alpha,
	\end{equation}
	\begin{equation}
	\psi\iota_1=\kappa,\,\,\psi\iota_2=\beta,\,\,\lambda\psi=\pi_1,\,\,\alpha\psi=\pi_2.
	\end{equation}
\end{lem}

\begin{remark}
	As Lemmas 1.6, 1.8, and 1.10 show, a split extension of magmas is nothing but a semidirect product diagram up to an isomorphism (cf. Remark 1.9). Using this, note:
	\begin{itemize}
		\item [(a)] The formulas $(11)$-$(17)$ are straightforward translations of the formulas $(3)$-$(9)$. However, the formulas $(11)$-$(17)$ are not logically independent from each other. In particular, $(14)$ immediately implies the first equalities of $(11)$ and $(12)$. Note also, that $(16)$ can be deduced from $(13)$, $(15)$, and $(17)$.
		\item [(b)] The counterpart of the addition formula from Definition 1.2(a) for the addition in $A$ is
		\begin{equation}
		(\kappa(x)+\beta(b))+(\kappa(x')+\beta(b'))=\kappa(x+bx')+\beta(b+b'),
		\end{equation}
		which just repeats Lemma 1.8. The reason why it works as such counterpart is that every element of $A$ can be uniquely presented as a sum of an element from $\kappa(X)$ and an element from $\beta(B)$; explicitly,  $a=(\kappa\lambda+\beta\alpha)(a)=\kappa\lambda(a)+\beta\alpha(a)$, for each $a\in A$.
	\end{itemize}
\end{remark}

Our last lemma of this section collects purely categorical properties of a split extension (10).
\begin{lem}
	\begin{itemize}
		\item [(a)] $\kappa$ and $\beta$ are jointly strongly epimorphic in the category of magmas;
		\item [(b)] $\lambda$ and $\alpha$ form a product diagram in the category of sets;
		\item [(c)] $\kappa$ is a kernel of $\alpha$ in the category of magmas;
		\item [(d)] $\alpha$ is a cokernel of $\kappa$ in the category of magmas.
	\end{itemize}
\end{lem}
\begin{proof} All these assertions easily follow from Lemmas 1.6 and 1.8. Nevertheless let us give a direct rather categorical proof of (d), which we will also need later. We have to show that, for every magma $C$ and magma homomorphism $s:A\to C$ with $s\kappa=0$, there exists a magma homomorphism $t:B\to C$ with $t\alpha=s$. For that, we take $t=s\beta$ and calculate:
	\begin{equation*}
	t\alpha=s\beta\alpha=s\beta\alpha(\kappa\lambda+\beta\alpha)=s\beta\alpha\kappa\lambda+s\beta\alpha\beta\alpha=0+s\beta\alpha=s\kappa\lambda+s\beta\alpha=s(\kappa\lambda+\beta\alpha)
	\end{equation*}
	$=s$. Equivalently, the same could be deduced from (a) without using $\lambda$, making the argument even `more categorical'.
\end{proof}
\begin{remark}
In \cite{[MFM]} N. Martins-Ferreira and A. Montoli  considered the class $S$ of Schreier split epimorphisms of unitary magmas, whose category can be identified (up to a canonical equivalence) with the category of diagrams $(10)$ satisfying conditions $(11)$ to $(14)$ in this paper. As observed in \cite{[MFM]} (Proposition 2.5) any J\'onsson-Tarski variety is $S$-protomodular in the sense of \cite{[BMMS]}, for $S$ the class of Schreier split epimorphisms. In particular, this fact implies that any J\'onsson-Tarski variety satisfies the $S$-relative version of the Split Short Five Lemma. Adding conditions $(15)$ to $(17)$ does not change that of course. However, in Section 4.4 below, a different approach will be used and some further observations on the validity of the Split Short Five Lemma in our context will be made.
\end{remark}

\section{The equivalence}

Let us present a split extension $(10)$ as the seven-tuple $E=(B,X,A,\alpha,\beta,\kappa,\lambda)$. This order of letters has a good reason:
\begin{itemize}
	\item at the level of objects, we think of $E$ as a split extension of $B$ with kernel $X$, and of $A$ as the object part of the extension;
	\item the pair $(\alpha,\beta)$ is what makes $A$ a split extension of $B$;
	\item the pair $(\kappa,\lambda)$ completes the structure by making $X$ the kernel of the extension.
\end{itemize}

\begin{defi}
	Let $E=(B,X,A,\alpha,\beta,\kappa,\lambda)$ and $E'=(B',X',A',\alpha',\beta',\kappa',\lambda')$ be split extensions of magmas. A morphism $E\to E'$ is a diagram morphism formed by magma homomorphisms, that is, it is a triple $(f,u,p)$ such that:
	\begin{itemize}
		\item [(a)] $f:B\to B'$, $u:X\to X'$, and $p:A\to A'$ are magma homomorphisms;
		\item [(b)] the diagram
		\begin{equation}\label{split-morphism}
		\xymatrix@=30pt{
			X  \ar[d]_u \ar[r]_-{\kappa}   & A  \ar[r]_-{\alpha} \ar@<-1ex>[l]_-{\lambda} \ar[d]_{p}  & B \ar@<-1ex>[l]_-{\beta}  \ar[d]^f \\X'  \ar[r]_-{\kappa'}   & A' \ar[r]_-{\alpha'} \ar@<-1ex>[l]_-{\lambda'}   & B' \ar@<-1ex>[l]_-{\beta'}
		}
		\end{equation}
		reasonably commutes, in the sense that
		\begin{equation}
		p\kappa=\kappa'u,\,\,p\beta=\beta'f,\,\,\lambda'p=u\lambda,\,\,\alpha'p=f\alpha.
		\end{equation}
	\end{itemize}
	The category whose objects are all split extensions of magmas and morphisms are defined above and composed componentvise will be called the category of split extensions of magmas and denoted by $\mathsf{SplExt}$.
\end{defi}

\begin{remark} Again (cf. Remark 1.11(a)), while requiring the four equalities (23) is natural of course, these equalities are not logically independent. In particular, the first two of them are equivalent to the last two. Indeed:
	\begin{itemize}
		\item Assuming $p\kappa=\kappa'u$ and $p\beta=\beta'f$, we have:
		\begin{equation*}
		\lambda'p=\lambda'p(\kappa\lambda+\beta\alpha)=\lambda'(p\kappa\lambda+p\beta\alpha)=\lambda'(\kappa'u\lambda+\beta'f\alpha)=u\lambda,
		\end{equation*}
		where the last equality follows from $(14)$ applied to $E'$;
		\begin{equation*}
		\alpha'p=\alpha'p(\kappa\lambda+\beta\alpha)=\alpha'p\kappa\lambda+\alpha'p\beta\alpha=\alpha'\kappa'u\lambda+\alpha'\beta'f\alpha=0+f\alpha=f\alpha.
		\end{equation*}
		\item On the other hand, assuming $\lambda'p=u\lambda$ and $\alpha'p=f\alpha$, we have:
		\begin{equation*}
		p\kappa=(\kappa'\lambda'+\beta'\alpha')p\kappa=\kappa'\lambda'p\kappa+\beta'\alpha'p\kappa=\kappa'u\lambda\kappa+\beta'f\alpha\kappa=\kappa'u+0=\kappa'u;
		\end{equation*}
		\begin{equation*}
		p\beta=(\kappa'\lambda'+\beta'\alpha')p\beta=\kappa'\lambda'p\beta+\beta'\alpha'p\beta=\kappa'u\lambda\beta+\beta'f\alpha\beta=0+\beta'f=\beta'f.
		\end{equation*}
	\end{itemize}
\end{remark}

\begin{lem}
	Let $E=(B,X,A,\alpha,\beta,\kappa,\lambda)$ and $E'=(B',X',A',\alpha',\beta',\kappa',\lambda')$ be split extensions of magmas with $X$ and $X'$ being equipped with the associated actions of $B$ and $B'$, respectively. If $(f,u,p):E\to E'$ is a morphism of split extensions, then for all $b\in B$ and $x\in X$, we have
	\begin{equation}
	u(bx)=f(b)u(x).
	\end{equation}
\end{lem}
\begin{proof}
	We have:
	\begin{equation*}
	u(bx)=u\lambda(\beta(b)+\kappa(x))
	\end{equation*}
	(by Definition 1.5)
	\begin{equation*}
	=\lambda'p(\beta(b)+\kappa(x))
	\end{equation*}
	(by the third equality in (23))
	\begin{equation*}
	=\lambda'(p\beta(b)+p\kappa(x))
	\end{equation*}
	(since $p$ is a magma homomorphism)
	\begin{equation*}
	=\lambda'(\beta'f(b)+\kappa'u(x))
	\end{equation*}
	(by the first two equalities of (23))
	\begin{equation*}
	=f(b)u(x)
	\end{equation*}
	(by Definition 1.5 applied to $E'$).
\end{proof}

\begin{remark} Rephrasing the first sentence of Remark 1.11, we can say now:
	
	As follows from Lemmas 1.6 and 1.8, the maps $\varphi$ and $\psi$ involved in diagram (18) determine inverse to each other isomorphisms $(1_B,1_X,\varphi)$ and $(1_B,1_X,\psi)$ between $E=(B,X,A,\alpha,\beta,\kappa,\lambda)$ and the bottom row of (18) seen as another split extension.
	
	Furthermore, now we can add:
	
	Applying Lemma 2.3 to these isomorphisms, we see that the action of $B$ on $X$ associated to $E$ is the same as the action $h$ associated to that bottom row. This, however, can be deduced also from Definition 1.5 directly. Indeed, using Definition 1.5, we obtain $h(b,x)=\pi_1(\iota_2(b)+\iota_1(x))=\pi_1((0,b)+(x,0))=\pi_1(bx,b)=bx$.
\end{remark}

\begin{lem}
	Let $E=(B,X,A,\alpha,\beta,\kappa,\lambda)$ and $E'=(B',X',A',\alpha',\beta',\kappa',\lambda')$ be split extensions of magmas, with $X$ and $X'$ equipped with the associated actions of $B$ and $B'$ respectively. For magma homomorphisms $f:B\to B'$, $u:X\to X'$, the following conditions are equivalent:
	\begin{itemize}
		\item [(a)] there exists a magma homomorphism $p:A\to A'$ such that $(f,u,p):E\to E'$ is a morphism of split extensions;
		\item [(b)] there exists a unique magma homomorphism $p:A\to A'$ such that $(f,u,p):E\to E'$ is a morphism of split extensions;
		\item [(c)] for all $b\in B$ and $x\in X$, we have
		\begin{equation}
		u(bx)=f(b)u(x).
		\end{equation}
	\end{itemize}
\end{lem}
\begin{proof}
	(a)$\Leftrightarrow$(b) immediately follows from any of the two first assertions of Lemma 1.12.
	
	(a)$\Rightarrow$(c) follows from Lemma 2.3.
	
	(c)$\Rightarrow$(a): Putting $p=\kappa'u\lambda+\beta'f\alpha$, we calculate:
	\begin{equation*}
	p\kappa=(\kappa'u\lambda+\beta'f\alpha)\kappa=\kappa'u\lambda\kappa+\beta'f\alpha\kappa=\kappa'u+0=\kappa'u,
	\end{equation*}
	\begin{equation*}
	p\beta=(\kappa'u\lambda+\beta'f\alpha)\beta=\kappa'u\lambda\beta+\beta'f\alpha\beta=0+\beta'f=\beta'f.
	\end{equation*}
	Having in mind Remark 2.2, this would complete the proof if we knew that $p$ is a magma homomorphism. However, we have
	\begin{equation*}
	p(a+a')=p((\kappa\lambda(a)+\beta\alpha(a))+(\kappa\lambda(a')+\beta\alpha(a')))
	\end{equation*}
	(see Remark 1.11(b))
	\begin{equation*}
	=p(\kappa(\lambda(a)+\alpha(a)\lambda(a'))+\beta(\alpha(a)+\alpha(a')))
	\end{equation*}
	(by (21))
	\begin{equation*}
	=\kappa'u\lambda(\kappa(\lambda(a)+\alpha(a)\lambda(a'))+\beta(\alpha(a)+\alpha(a')))+\beta'f\alpha(\kappa(\lambda(a)+\alpha(a)\lambda(a'))+\beta(\alpha(a)
	\end{equation*}
	\begin{equation*}
	+\alpha(a')))
	\end{equation*}
	(by definition of $p$)
	\begin{equation*}
	=\kappa'u(\lambda(a)+\alpha(a)\lambda(a'))+\beta'f(\alpha(a)+\alpha(a'))
	\end{equation*}
	(by (14), the fact that $\alpha$ is a magma homomorphism, and the second equalities of (11) and (12))
	\begin{equation*}
	=\kappa'(u\lambda(a)+f(\alpha(a))u(\lambda(a')))+\beta'(f\alpha(a)+f\alpha(a'))
	\end{equation*}
	(since $u$ and $f$ are magma homomorphisms, and using (25))
	\begin{equation*}
	=(\kappa'u\lambda(a)+\beta'f\alpha(a))+(\kappa'u\lambda(a')+\beta'f\alpha(a'))
	\end{equation*}
	(by (21) applied to $E'$)
	\begin{equation*}
	=p(a)+p(a')
	\end{equation*}
	(by definition of $p$).
\end{proof}

\begin{remark}
	We could make the last calculation look shorter by using the semidirect product notation. In that notation $p$ would be defined by $p(x,b)=(u(x),f(b))$, and the calculation would become
	\begin{equation*}
	p((x,b)+(x',b'))=p(x+bx',b+b')=(u(x+bx'),f(b+b'))
	\end{equation*}
	\begin{equation*}
	=(u(x)+f(b)u(x'),f(b)+f(b'))=(u(x),f(b))+(u(x'),f(b'))=p(x,b)+p(x'+b').
	\end{equation*}
\end{remark}

\begin{defi} Let $h:B\times X\to X$ and $h':B'\times X'\to X'$ be actions of magmas, $B$ on $X$ and $B'$ on $X'$, respectively. A morphism $(B,X,h)\to(B',X',h')$ is a pair $(f,u)$ such that:
	\begin{itemize}
		\item [(a)] $f:B\to B'$ and $u:X\to X'$ are magma homomorphisms;
		\item [(b)] the diagram
		\begin{equation}
		\xymatrix{ B \times X \ar[d]_{f \times u} \ar[r]^-h & X \ar[d]^u \\
			B' \times X' \ar[r]_-{h'} & X'
		}
		\end{equation}
		commutes, or, equivalently, in the notation of Lemma 2.3, the equality $(24)$ holds for all $b\in B$ and $x\in X$. The resulting category of actions of magmas will be denoted by $\mathsf{Act}$.
	\end{itemize}
\end{defi}

From Lemma 2.5 and previous results, we obtain:
\begin{teo} There is an equivalence between the category $\mathsf{SplExt}$ of split extensions of magmas and the category $\mathsf{Act}$ of actions of magmas, constructed as follows:
	\begin{itemize}
		\item [(a)] the functor $\mathsf{SplExt}\to \mathsf{Act}$ carries a morphism
		\begin{equation*}
		(f,u,p):(B,X,A,\alpha,\beta,\kappa,\lambda)\to(B',X',A',\alpha',\beta',\kappa',\lambda')
		\end{equation*}
		of split extensions to the morphism $(f,u):(B,X,h)\to(B',X',h')$ of the associated actions;
		\item [(b)] the functor $\mathsf{Act}\to \mathsf{SplExt}$ carries a morphism
		\begin{equation*}
		(f,u):(B,X,h)\to(B',X',h')
		\end{equation*}
		of actions to the corresponding morphism $(f,u,p)$ between semidirect product diagrams seen as split extensions, where $p$ is defined by
		\begin{equation*}
		p(x,b)=(u(x),f(b)).
		\end{equation*}
	\end{itemize}
\end{teo}

\section{Composition of split extensions}

Consider a diagram of the form
\begin{equation}\label{canonical-iso}
\xymatrix@= 30pt{& Z\ar@<-0.5ex>[d]_{\mu'} \ar[r]_-{\alpha'} & Y \ar@<-0.5ex>[d]_{\mu}\ar@<-1ex>[l]_-{\beta'}\\
	X \ar[r]_-{\kappa}   & A \ar@<-0.5ex>[u]_{\nu'} \ar[r]_-{\alpha} \ar@<-1ex>[l]_-{\lambda} & B \ar@<-0.5ex>[u]_{\nu} \ar@<-1ex>[l]_-{\beta} \ar[r]_-{\gamma} & D \ar@<-1ex>[l]_-{\delta} }
\end{equation}
in which:
\begin{itemize}
	\item $(B,X,A,\alpha,\beta,\kappa,\lambda)$ and $(D,Y,B,\gamma,\delta,\mu,\nu)$ are split extensions of magmas;
	\item $Z=\alpha^{-1}(Y)$ with $\mu'$ being the inclusion map and $\alpha'$ and $\beta'$ being induced by $\alpha$ and $\beta$, respectively;
	\item $\nu'$ is defined by $\nu'(a)=\kappa\lambda(a)+\beta\mu\nu\alpha(a)$.
\end{itemize}
\begin{defi}
	In the notation above, with $E=(B,X,A,\alpha,\beta,\kappa,\lambda)$ and $F=(D,Y,B,\gamma,\delta,\mu,\nu)$, we will say that the split extensions $F$ and $E$ compose and call $(F,E)$ a composable pair of split extensions, if the seven-tuple $G=(D,Z,A,\gamma\alpha,\beta\delta,\\
	\mu',\nu')$ is a split extension (of magmas). If it is the case, we shall write $G=FE$ and call $G$ the composite of $F$ and $E$.
\end{defi}

Thanks to Theorem 2.8, diagram (27) can be identified, up to isomorphism, with
\begin{equation}\label{canonical-iso}
\xymatrix@= 30pt{& X\rtimes(Y\times \{0\})\ar@<-0.5ex>[d]_{\mu'} \ar[r]_-{\alpha'} & Y \ar@<-0.5ex>[d]_{\mu}\ar@<-1ex>[l]_-{\beta'}\\
	X \ar[r]_-{\kappa}   & X\rtimes(Y\rtimes D) \ar@<-0.5ex>[u]_{\nu'} \ar[r]_-{\alpha} \ar@<-1ex>[l]_-{\lambda} & Y\rtimes D \ar@<-0.5ex>[u]_{\nu} \ar@<-1ex>[l]_-{\beta} \ar[r]_-{\gamma} & D \ar@<-1ex>[l]_-{\delta} }
\end{equation}
where:
\begin{itemize}
	\item we assume that $D$ acts on $Y$ and $Y\rtimes D$ act on $X$, that is, some actions of $D$ on $Y$ and of $Y\rtimes D$ on $X$ are given;
	\item since $Y\times \{0\}$ is a submagma of $Y\rtimes D$, the action of $Y\rtimes D$ on $X$ induces the action of $Y\rtimes \{0\}$ on $X$;
	\item the maps involved are defined in the straightforward way, that is, by
	\begin{equation*}
	\alpha(x,(y,d))=(y,d),\,\,\beta(y,d)=(0,(y,d)),\,\,\kappa(x)=(x,(0,0)),\,\,\lambda(x,(y,d))=x,
	\end{equation*}
	\begin{equation*}
	\gamma(y,d)=d,\,\,\delta(d)=(0,d),\,\,\mu(y)=(y,0),\,\,\nu(y,d)=y,
	\end{equation*}
	with $\mu'$ being the inclusion map, $\alpha'$ and $\beta'$ induced by $\alpha$ and $\beta$, respectively, and
	\begin{equation*}
	\nu'(x,(y,d))=\kappa\lambda(x,(y,d))+\beta\mu\nu\alpha(x,(y,d))=(x,(0,0))+(0,(y,0))=(x,(y,0)).
	\end{equation*}
\end{itemize}

Accordingly, the seven-tuple $(D,Z,A,\gamma\alpha,\beta\delta,\mu',\nu')$ becomes now $(D,X\rtimes(Y\times \{0\}),X\rtimes(Y\times D),\gamma\alpha,\beta\delta,\mu',\nu')$.
\begin{lem}
	The seven-tuple $(D,X\rtimes(Y\times \{0\}),X\rtimes(Y\times D),\gamma\alpha,\beta\delta,\mu',\nu')$ satisfies the counterparts of equalities $(11)$-$(14)$, $(16)$, and $(17)$.
\end{lem}
\begin{proof}
These counterparts are, respectively:
\begin{equation}
    \nu'\mu'=1,\,\,\gamma\alpha\beta\delta=1,
\end{equation}
\begin{equation}
\nu'\beta\delta=0,\,\,\gamma\alpha\mu'=0,
\end{equation}
\begin{equation}
\mu'\nu'+\beta\delta\gamma\alpha=1,
\end{equation}
\begin{equation}
\nu'(\mu'(x,(y,0))+\beta\delta(d))=(x,(y,0)),
\end{equation}
\begin{equation}
\mu'(x',(y',0))+((x,(y,d))+\beta\delta(d'))=(\mu'(x',(y',0))+(x,(y,d)))+\beta\delta(d'),
\end{equation}
\begin{equation}
(x,(y,d))+(\mu'(x',(y',0))+\beta\delta(d'))=((x,(y,d))+\mu'(x',(y',0)))+\beta\delta(d'),
\end{equation}
required for all $x$ and $x'$ in $X$, $y$ and $y'$ in $Y$, and $d$ and $d'$ in $D$. Here (29) and (30) are obvious, and we have:
\begin{equation*}
(31):\,\,(\mu'\nu'+\beta\delta\gamma\alpha)(x,(y,d))=(x,(y,0))+(0,(0,d))=(x,(y,d)),
\end{equation*}
\begin{equation*}
(32):\,\,\nu'(\mu'(x,(y,0))+\beta\delta(d))=\nu'((x,(y,0))+(0,(0,d)))=\nu'(x,(y,d))=(x,(y,0)),
\end{equation*}
\begin{equation*}
(33):\,\,\mu'(x',(y',0))+((x,(y,d))+\beta\delta(d'))=(x',(y',0))+((x,(y,d))+(0,(0,d'))
\end{equation*}
\begin{equation*}
=(x',(y',0))+((x,(y,d+d'))=(x'+(y',0)x,(y'+y,d+d'))
\end{equation*}
\begin{equation*}
=(x'+(y',0)x,(y'+y,d))+(0,(0,d'))=((x',(y',0))+(x,(y,d)))+(0,(0,d'))
\end{equation*}
\begin{equation*}
=(\mu'(x',(y',0))+(x,(y,d)))+\beta\delta(d'),
\end{equation*}
\begin{equation*}
(34):\,\,(x,(y,d))+(\mu'(x',(y',0))+\beta\delta(d'))=(x,(y,d))+((x',(y',0))+(0,(0,d')))
\end{equation*}
\begin{equation*}
=(x,(y,d))+(x',(y',d'))=(x+(y,d)x',(y+dy',d+d'))
\end{equation*}
\begin{equation*}
=(x+(y,d)x',(y+dy',d))+(0,(0,d'))=((x,(y,d))+(x',(y',0)))+(0,(0,d'))
\end{equation*}
\begin{equation*}
=((x,(y,d))+\mu'(x',(y',0)))+\beta\delta(d').
\end{equation*}
\end{proof}
\begin{lem}
	The following conditions are equivalent:
	\begin{itemize}
		\item [(a)] $(D,X\rtimes(Y\times \{0\}),X\rtimes(Y\times D),\gamma\alpha,\beta\delta,\mu',\nu')$ is a split extension (of magmas);
		\item [(b)] $(D,X\rtimes(Y\times \{0\}),X\rtimes(Y\times D),\gamma\alpha,\beta\delta,\mu',\nu')$ satisfies the counterpart of equality $(15)$;
		\item [(c)] the equality
		\begin{equation}
		(y,0)((0,d)x)=(y,d)x
		\end{equation}
		holds for all $x$ in $X$, $y$ in $Y$, and $d$ in $D$.
	\end{itemize}
\end{lem}
\begin{proof}
	(a)$\Leftrightarrow$(b) follows from Definition 1.4 and Lemma 3.2.
	
	(b)$\Leftrightarrow$(c):
	To say that $(D,X\rtimes(Y\times \{0\}),X\rtimes(Y\times D),\gamma\alpha,\beta\delta,\mu',\nu')$ satisfies the counterpart of equality $(15)$ is to say that
	\begin{equation}
	\mu'(x,(y,0))+(\beta\delta(d)+(x',(y',d')))=(\mu'(x,(y,0))+\beta\delta(d))+(x',(y',d')),
	\end{equation}
	or, equivalently,
	\begin{equation*}
	(x,(y,0))+((0,(0,d))+(x',(y',d')))=((x,(y,0))+(0,(0,d)))+(x',(y',d')),
	\end{equation*}
	holds for all $x$ and $x'$ in $X$, $y$ and $y'$ in $Y$, and $d$ and $d'$ in $D$. We have
	\begin{equation*}
	(x,(y,0))+((0,(0,d))+(x',(y',d')))=(x,(y,0))+((0,d)x',(dy',d+d'))
	\end{equation*}
	\begin{equation*}
	=(x+(y,0)((0,d)x'),(y+dy',d+d'))
	\end{equation*}
	and
	\begin{equation*}
	((x,(y,0))+(0,(0,d)))+(x',(y',d'))=(x,(y,d))+(x',(y',d'))
	\end{equation*}
	\begin{equation*}
	=(x+(y,d)x',(y+dy',d+d')),
	\end{equation*}
	and these two expressions are equal for all $x$ and $x'$ in $X$, $y$ and $y'$ in $Y$, and $d$ and $d'$ in $D$ if and only if $(y,0)((0,d)x)=(y,d)x$ for all $x$ in $X$, $y$ in $Y$, and $d$ in $D$.
\end{proof}

Translating Lemma 3.3 into the context of general split extensions with the notation of Definition 3.1, we obtain:
\begin{cor}
	Given split extensions $E=(B,X,A,\alpha,\beta,\kappa,\lambda)$ and $F=(D,Y,B,\gamma,\\
	\delta,\mu,\nu)$, the following conditions are equivalent:
	\begin{itemize}
		\item [(a)] $(F,E)$ is a composable pair;
		\item [(b)] the equality
		\begin{equation}
		\mu'(z)+(\beta\delta(d)+a)=(\mu'(z)+\beta\delta(d))+a,
		\end{equation}
		holds for all $z\in Z$, $d\in D$, and $a\in A$;
		\item [(c)] the equality
		\begin{equation}
		\mu(y)(\delta(d)x)=(\mu(y)+\delta(d))x,
		\end{equation}
		holds for all $y\in Y$, $d\in D$, and $x\in X$; this formula uses the action $B$ on $X$ given in Definition 1.5.\qed
	\end{itemize}
\end{cor}
\begin{defi}
	Let $B$ and $X$ be magmas. An action of $B$ on $X$ will be called firm if
	\begin{equation}
	b'(bx)=(b'+b)x,
	\end{equation}
	for all $b,\,b'\in B$ and $x\in X$. Accordingly, a split extension of magmas will be called firm if its associated action is firm.
\end{defi}

Corollary 3.4 immediately implies
\begin{cor}
	Given split extensions $E=(B,X,A,\alpha,\beta,\kappa,\lambda)$ and $F=(D,Y,B,\gamma,\\
	\delta,\mu,\nu)$, as in diagram $(27)$, the pair $(F,E)$ is composable whenever $E$ is firm.\qed
\end{cor}

Furthermore, we have:
\begin{teo}
	If the split extensions $E=(B,X,A,\alpha,\beta,\kappa,\lambda)$ and $F=(D,Y,B,\gamma,\\
	\delta,\mu,\nu)$, as in diagram $(27)$, are firm, then so is their composite $FE= (D,Z,A,\gamma\alpha,\\
	\beta\delta,\mu',\nu')$.
\end{teo}
\begin{proof}
	Again, without loss of generality, we can replace diagram (27) with diagram (28), and then we will have to prove that
	\begin{equation}
		d'(d(x,(y,0)))=(d'+d)(x,(y,0)),
	\end{equation}
	for all $d$ and $d'$ in $D$, $x$ in $X$, and $y$ in $Y$. Using Definition 1.5 and the fact that $E$ and $F$ are firm, we calculate:
	\begin{eqnarray*}
	d'(d(x,(y,0)))& = & d'(\nu'(\beta\delta(d)+\mu'(x,(y,0)))) \\ &=& d'(\nu'((0,(0,d))+(x,(y,0)))) \\
	& = & d'(\nu'((0,d)x,(dy,d))) \\ &=& d'((0,d)x,(dy,0)) \\ &=& ((0,d')((0,d)x),(d'(dy),0)) \\
	& = & (((0,d)+(0,d'))x,((d'+d)y,0))\\ &=&((0,d+d')x,((d'+d)y,0))\\ &=& (d'+d)(x,(y,0))
	\end{eqnarray*}
\end{proof}

\section{Conclusions and additional remarks}

\subsection{Cotranslations = Cartesian liftings = Pullbacks} Given a split extension $E=(B,X,A,\alpha,\beta,\kappa,\lambda)$ (of magmas) and a magma homomorphism $f:B'\to B$ (we could also consider the case where $f$ merely preserves zero, but let us omit that), we can construct, uniquely up to isomorphism, a diagram of the form
\begin{equation}
\xymatrix@=30pt{
	X  \ar@{=}[d] \ar[r]_-{\kappa'}   & A'  \ar[r]_-{\alpha'} \ar@<-1ex>[l]_-{\lambda'} \ar[d]_{p}  & B' \ar@<-1ex>[l]_-{\beta'}  \ar[d]^f \\X  \ar[r]_-{\kappa}   & A \ar[r]_-{\alpha} \ar@<-1ex>[l]_-{\lambda}   & B, \ar@<-1ex>[l]_-{\beta}
}
\end{equation}
with the top row being a split extension, which we will denote by $E_f$, and $(f,1,p):E_f\to E$ being a morphism in $\mathsf{SplExt}$. Moreover, this diagram will satisfy the usual universal property. This follows from Theorem 2.8 and the fact that the equivalence described there agrees with the forgetful functors to the category of magmas, both defined by $(B,...)\mapsto B$. This construction is, of course, a \textit{cotranslation} of $E$ along $f$ in the sense of Yoneda \cite{[Y]}, which is the same as to say that it is a cartesian lifting of $f$ in the sense of the theory of Grothendieck fibrations. In particular, $\alpha'$ is a pullback of $\alpha$ along $f$. The reason why this is easy is that so is making cartesian liftings for the fibration of actions: given an action $h:B\times X\to X$ and $f$ above, just use the composite $h(f\times X):B'\times X\to X$. This also tells us that $E'$ \textit{is firm whenever so is} $E$. In the language of semidirect products the diagram (41) is (up to canonical isomorphisms) the same as
\begin{equation}
\xymatrix@=30pt{
	X  \ar@{=}[d] \ar[r]_-{\iota_1'}   & X\rtimes B'  \ar[r]_-{\pi_2'} \ar@<-1ex>[l]_-{\pi_1'} \ar[d]_{X\times f}  & B' \ar@<-1ex>[l]_-{\iota_2'}  \ar[d]^f \\X  \ar[r]_-{\iota_1}   & X\rtimes B \ar[r]_-{\pi_2} \ar@<-1ex>[l]_-{\pi_1}   & B, \ar@<-1ex>[l]_-{\iota_2}
}
\end{equation}
in obvious notation, with $B'$ acting on $X$ via $b'x=f(b')x$. Note also, that these cartesian liftings do not change the kernel $X$ and so instead of the equivalence $\mathsf{SplExt}\sim\mathsf{Act}$ of Theorem 2.8 we could use its restriction $\mathsf{SplExt}(-,X)\sim\mathsf{Act}(-,X)$ on split extensions and actions with the fixed ``$X$ part''.
\subsection{Translations = Cocartesian liftings = ``Pushforwards''} It is well known, already for ordinary split extensions of groups, that the functors associating to split extensions their kernels are neither fibrations nor opfibrations. However (which is also well known for groups), as follows from Theorem 2.8, the forgetful functor $\mathsf{SplExt}(B,-)\to \mathsf{Act}(B,-)$ is a category equivalence. This gives us trivial cocartesian and trivial cartesian liftings of morphisms in $\mathsf{Act}(B,-)$ with respect to that functor, the same as translations and cotranslations in the sense of Yoneda \cite{[Y]} with respect to the span, which Yoneda would write as $\mathsf{SplExt}(B,-)\to 1\times\mathsf{Act}(B,-)$ (if he would use our notation for the categories involved, with $1$ being a terminal category). In some recent papers similar kinds of translations are called ``pushforwards''(e.g.
\cite{[CMM]}, \cite{[MMS2]}). There is also an obvious relationship between translations here and cotranslations in the sense of 4.1, which we shall not describe since it copies what is well known for groups.
\subsection{What are `good' split epimorphisms of magmas?} One might think of three such classes:
\begin{itemize}
	\item The class $\mathcal{E}$ of all magma homomorphisms $\alpha$, for which there exists a split extension $E=(B,X,A,\alpha,\beta,\kappa,\lambda)$ with same $\alpha$.
	\item The class $\mathcal{E}'$ defined in the same way except that $E$ is required to be firm in the sense of Definition 3.5.
	\item The class $\mathcal{E}''$ defined in the same way except that $E$ is required not only to be firm, but also to satisfy the equality
	\begin{equation}
	b(x+x')=bx+bx',
	\end{equation}
	for all $b\in B$ and $x,\,x'\in X$.
\end{itemize}
We observe:
\begin{itemize}
	\item [(a)] As in fact explained in Subsection 4.1, Theorem 2.8 implies that the classes $\mathcal{E}$ and $\mathcal{E}'$ are pullback stable. The same is true for $\mathcal{E}''$: for, note that, in the notation of 4.1, if the action $h:B\times X\to X$ satisfies (43), then the action $h(f\times X):B'\times X\to X$ also does.
	\item [(b)] It is a triviality to construct an example where condition 3.4(c) does not hold. Therefore Corollary 3.4 implies that the class $\mathcal{E}$ is not closed under composition. On the other hand, as follows from Corollary 3.6 and Theorem 3.7, the class $\mathcal{E}'$ is closed under composition.
	\item [(c)] The class $\mathcal{E}''$ is also closed under composition. To prove this, we have to prove that, in the situation of Definition 3.1 with $E,\,F\in \mathcal{E}''$, the action of $D$ on $X\rtimes(Y\times{0})$ associated to $G=FE$ satisfies (43). That is, we need to prove that, under the assumptions above,
	\begin{equation}
	d((x,(y,0))+(x',(y',0)))=d(x,(y,0))+d(x',(y',0))
	\end{equation}
	for all $d$ in $D$, $x$ and $x'$ in $X$, and $y$ and $y'$ in $Y$.
	Going back to the calculation in the proof of Theorem 3.7, and having in mind that
	\begin{equation*}
	(x,(y,0))+(x',(y',0))=(x+(y,0)x',(y+y',0)),
	\end{equation*}
	we see that to prove (44) is to prove
	\begin{equation}
	((0,d)(x+(y,0)x'),(d(y+y'),0))=((0,d)x,(dy,0))+((0,d)x',(dy',0)).
	\end{equation}
	We have
	\begin{eqnarray*}
	((0,d)x,(dy,0))+((0,d)x',(dy',0))& = & ((0,d)x+(dy,0)((0,d)x'),(dy+dy',0)) \\
	& = & ((0,d)x+((dy,0)+(0,d))x',(d(y+y'),0))\\ &=& ((0,d)x+(dy,d)x',(d(y+y'),0)) \\
	& = & ((0,d)x+((0,d)+(y,0))x',(d(y+y'),0)) \\  & = & ((0,d)x+(0,d)((y,0)x'),(d(y+y'),0)) \\
	& = &((0,d)(x+(y,0)x'),(d(y+y'),0)),
	\end{eqnarray*}
	as desired. Note that this calculation used (39), which suggests that requiring (43) without (39) would be less natural than requiring (39) without (43).
\end{itemize}
\subsection{Split Short Five Lemma}
In our context the Split Short Five Lemma should say that, in diagram (22), $p$ is an isomorphism whenever so are $f$ and $u$. If one uses this formulation, then the Split Short Five Lemma immediately follows from Theorem 2.8. But what if we make the assumptions on $p$ weaker omitting some of the equalities in (23)? Thanks to Theorem 2.8, this question is equivalent to the following one:
let
\begin{equation}\label{canonical-iso}
\xymatrix@=30pt{
	X  \ar@{=} [d] \ar[r]_-{\iota_1}   & X \rtimes B  \ar[r]_-{\pi_2} \ar@<-1ex>[l]_-{\pi_1} \ar[d]_p  & B \ar@<-1ex>[l]_-{\iota_2}  \ar@{=}[d] \\X  \ar[r]_-{\iota_1}   & X \rtimes B \ar[r]_-{\pi_2} \ar@<-1ex>[l]_-{\pi_1} & B \ar@<-1ex>[l]_-{\iota_2}
}
\end{equation}
be a diagram with the rows being semidirect diagrams of magmas (although the actions of $B$ on $X$ in them might be different from each other), and $p$ being a magma homomorphism; which of the equalities (23) do we need to prove that $p$ is an isomorphism?

Let us consider three cases:
\begin{itemize}
	\item [(a)] Suppose $\pi_1p=\pi_1$ and $\pi_2p=\pi_2$. This case is trivial: by the universal property of the bottom product diagram $p$ becomes a bijection (in fact the identity map) of the underlying sets, and since it was required to be a magma homomorphism, this makes it an isomorphism.
	\item [(b)] Suppose $p\iota_1=\iota_1$ and $p\iota_2=\iota_2$. This case is ``next to trivial'': we have $p(x,b)=p((x,0)+(0,b))=p(x,0)+p(0,b)=p\iota_1(x)+p\iota_2(b)=\iota_1(x)+\iota_2(b)$\\$=(x,0)+(0,b)=(x,b)$, and so $p$ is an isomorphism again.
	\item [(c)] Suppose $p\iota_1=\iota_1$ and $\pi_2p=\pi_2$. When $X$ is a commutative monoid and both actions of $B$ on $X$ are trivial, for any homomorphism $s:B\to X$, one could define $p$ by $p(x,b)=(x+s(b),b)$. It is then very easy to construct an example where such $p$ is not an isomorphism.
\end{itemize}

\subsection{Involving additional operations}
One could try to repeat our story involving additional operations on magmas, satisfying rather strong equational axioms. One seemingly straightforward way to do it is to follow Martins-Ferreira, Montoli, and Sobral \cite{[MMS]}, where the axioms are monoid counterparts of Porter's axioms for what he called ``groups with operations'' \cite{[P]}.

\subsection{``Internalization''} All definitions, constructions, and results of Sections 2 and 3, and everything we already said in this Section (except 4.4(c)) can be extended from ordinary magmas to internal magmas in a category $\mathcal{C}$ via the Yoneda embedding. Doing so it is better to assume, for simplicity, that $\mathcal{C}$ has finite products, although even that could be avoided (cf. the last paragraph of Section III.6 in \cite{[M]}). The reason for such assumption is that is preferable to have semidirect products of magmas inside $\mathcal{C}$. Let us also point out that constructing diagrams (27) and (28) internally in $\mathcal{C}$, obviously, one should:
\begin{itemize}
	\item Define $Z$, $\mu'$, and $\alpha'$ in (27) by requiring the square formed by them and $\alpha$ and $\mu$ to be a pullback square, and then make other replacements accordingly.
	\item Constructing diagram (28), define the actions of $Y$ on $X$ using the the action of $Y\rtimes D$ on $X$ and the morphism $\langle1,0\rangle:Y\to Y\rtimes D$; after that it is more convenient to write $X\rtimes Y$ instead of $X\rtimes(Y\times\{0\})$ and, again, make other replacements accordingly.
\end{itemize}

\subsection{Monoids and groups} Almost every word we say becomes well known when we replace magmas with monoids (see \cite{[BMMS]} and references therein). In particular, in the case of monoids we have:
\begin{itemize}
	\item The equalities (15)-(17) in Definition 1.4 hold automatically. The same is true for (39) and (43), and so all split extensions are firm, and the monoid counterparts of the classes $\mathcal{E}$, $\mathcal{E}'$, and $\mathcal{E}''$ of Subsection 4.3 coincide with each other.
	\item The split extensions in the sense of Definition 1.4 become nothing but Schreier split sequences in the sense of \cite{[BMMS]}.
\end{itemize}
In the classical case of groups, more ingredients of Definition 1.4 are redundant, and whole its data reduces to giving an arbitrary split epimorphism together with its splitting and kernel. The group-theoretic case of our Theorem 2.8 is nothing but a categorical formulation of what is well known in homological algebra as a first step towards cohomological description on group extensions.

\end{document}